\documentclass[10pt,twoside]{siamart1116}
\usepackage[english]{babel}
\usepackage{graphicx,epstopdf,epsfig}
\usepackage{amsfonts,epsfig,fancyhdr,graphics, hyperref,amsmath,amssymb}

\newcommand{\HE}{Namie of Handling Editor}
\newcommand{\DoS}{Month/Day/Year}
\newcommand{\DoA}{Month/Day/Year}
\newcommand{\CA}{Matja\v{z} Omladi\v{c}}

\newcommand{\Names}{Bojan Kuzma, Mitja Mastnak, Heydar Radjavi, and  Matja\v{z} Omladi\v{c}}
\newcommand{\Title}{Matrix groups}

\newtheorem{remark}[theorem]{Remark}

\newcommand{\reals}{\mathbb{R}}
\newcommand{\complex}{\mathbb{C}}
 
\renewcommand{\theequation}{\thesection.\arabic{equation}}
\renewtheorem{theorem}{Theorem}[section]

\begin{document}

\bibliographystyle{plain}

%  Leave these commented lines here
%\input{ELAheader-template.tex}
% ELA insert correct page number
\setcounter{page}{1}

\thispagestyle{empty}

%Insert the title of the paper
 \title{ On approximate and actual reducibility of matrix groups\thanks{Received
 by the editors on \DoS.
 Accepted for publication on \DoA. 
 Handling Editor: \HE. Corresponding Author: \CA.}}

\author{
Bojan Kuzma, \thanks{ Faculty of Mathematics, Natural Sciences and Information Technologies, University of Primorska, Koper, Slovenia (bojan.kuzma@famnit.upr.si). Supported in part by the Slovenian Research Agency, research program P1-0285 and research project N1-0210. }
\and
Mitja Mastnak, \thanks{ Department of Mathematics, Saint Mary's University, Halifax, Canada (mmastnak@cs.smu.ca). Supported by a grant from NSERC, Canada. }
\and
Matja\v{z} Omladi\v{c}, \thanks{ Department of Mathematics, Institute of Mathematics Physics and Mechanics, Ljubljana, Slovenia (matjaz@omladic.net). }
\and
Heydar Radjavi, \thanks{ Department of Pure Mathematics, University of Waterloo, Canada (hradjavi@uwaterloo.ca). }
}

\markboth{\Names}{\Title}

\maketitle

\begin{abstract}
We introduce the notions of $\varepsilon$-approximate fixed point and weak $\varepsilon$-approximate fixed point. We show that for a group of unitary matrices even the existence of a nontrivial weak $\varepsilon$-approximate fixed point for sufficiently small $\varepsilon$ gives an actual nontrivial common eigenvector. We give estimates for $\varepsilon$ in terms of the size $n$ of matrices and prove that the dependence is polynomial. Moreover, we show that the common eigenvector is polynomially close to the starting weak approximate fixed point. 
\end{abstract}



%%%%%%%%%%%%%%%%%%%%%%%%%%%%%%%%%%%%%%%%%%%%%%%%%%%%%%%%%%%%%
\section{Introduction} \label{intro-sec}
Let $\mathcal{G}$ be a group of complex matrices. An approximate fixed point $\xi\in\complex^n$ is defined to be a nonzero  vector for which the norm $\|G\xi-\xi\|$ is uniformly small for all $G\in\mathcal{G}$ in an appropriate sense. For a general group, given $\varepsilon >0$, one demands that $\|G\xi-\xi\| \leqslant\|G\|\, \|\xi\|\,\varepsilon$, for all $G\in\mathcal{G}$. (In this paper $\|\cdot\|$ will always mean $\|\cdot\|_2$ unless stated otherwise.) Since we will be concerned with groups of unitary matrices, we might as well assume $ \|G\|= \|\xi\|=1$ and simplify our definitions. The above definition then is equivalent to the requirement that the   inner product $\langle G\xi,\xi\rangle$ be uniformly close to 1. An obviously weaker condition is to require that the modulus $|\langle
G\xi,\xi\rangle|$ be uniformly close to 1. That kind of $\xi$ will then be called a weak approximate fixed point. The problem of determining groups for which the existence of weak approximate fixed points in the above sense implies the existence of a nontrivial invariant subspace of $\complex^n$ for $\mathcal{G}$ (i.e., reducibility of $\mathcal{G}$) doesn't seem to be easy. In this paper we impose additional hypotheses on $\mathcal{G}$ to get affirmative results.

\section{Preliminaries}
Let an integer $n>0$, a real number $\varepsilon>0$ and a group $\mathcal{G}\subseteq M_n(\complex)$ be given  (all the groups in this paper will consist of invertible complex matrices of order~$n$ so that they may be seen as subgroups of the general linear group). 

We call a point $\xi\in\complex^n$ of norm one {\emph{a fixed point of~$\mathcal{G}$} if $G\xi=\xi$ for each $G\in\mathcal{G}$, and} we call it an \emph{$\varepsilon$-approximate fixed point} of $\mathcal{G}$ if $\|G\xi-\xi\| \leqslant \varepsilon$ for all $G\in \mathcal{G}$. When $\varepsilon$ is understood from the context, we call $\xi$ shortly an \emph{approximate fixed point} of the group $\mathcal{G}$. {Recall that the existence of an approximate fixed point is closely related to Kazhdan property~$T$  introduced in a seminal paper of Kazhdan~\cite{Kazhd}; we refer to~\cite{Bek-Harp-Val} for a historical background and a wealth of information on the topic. In particular (see~\cite[Proposition 1.1.5]{Bek-Harp-Val}), if a compact unitary subgroup has an $\varepsilon<\sqrt{2}$ approximate fixed point~$\xi$, then it has a fixed point and we have the estimate $\|\xi-P\xi\|\le 1$ where $P$ is the  projection to the span of all fixed points (see~\cite[Proposition 1.1.9]{Bek-Harp-Val}). Clearly, $P\xi$ is a common eigenvector of ${\mathcal G}$, but note that in general its norm is less than one, so it does not qualify as a fixed point according to our definition above. }

We call $\xi\in\complex^n$ of norm one a \emph{weak  $\varepsilon$-approximate fixed point} of $\mathcal{G}$ if {$0<\varepsilon<1$ and}
\begin{equation}\label{weak appr fixed}
    |\langle G\xi,\xi \rangle| \geqslant 1- \varepsilon
\end{equation}
for all $G\in \mathcal{G}$. 
Again, when $\varepsilon$ is understood from the context, we call $\xi$ a \emph{weak approximate fixed point} of the group $\mathcal{G}$. The condition $\|\xi\|=1$ will be tacitly assumed for any fixed point $\xi$ in this paper in connection with either of these definitions, approximate or not, in any of the claims to follow.

The main result of our paper is that for $0<\varepsilon<\frac{1}{3600 n^{11}}$, any unitary subgroup $\mathcal{G}$ of $M_n(\complex)$ with a weak $\varepsilon$-approximate fixed point $\xi$ has a {fixed point} $\eta$ that is within $3600 n^{11} \varepsilon$ of $\xi$ (see Theorem \ref{main}).
The existence of an $\varepsilon>0$ for which existence of weak $\varepsilon$-approximate fixed points implies reducibility for unitary subgroups of $M_n(\complex)$ follows from a significantly more general result from \cite{KuMaOmRa}. There the authors study those  continuous multi-variable functions $$f\colon M_n(\complex)\times \dots \times M_n(\complex)\to [0,\infty),$$ that are reducing for unitary groups (in the sense that any unitary group $\mathcal{G}$ on which $f$ is identically $0$ is reducible).  They show, that for any fixed $n$, and for any such function $f$, there is an $\varepsilon>0$ such that any unitary group $\mathcal{G}$ on which $f$ is bounded by $\varepsilon$ %(i.e., the image of $f$ on $\mathcal{G}$ is contained in $[0,\varepsilon]$)
is also reducible.  The main ingredient of establishing this is applying the theory of collections of nonempty compact subsets of compact metric spaces equipped with the Hausdorff topology to the metric space of complex matrices with the usual topology. %The groups of unitary matrices are then seen as special cases of collections of subsets of the compact group of all unitaries.
The just mentioned result easily implies, that, for any $n$, there is an $\epsilon>0$ such that any unitary subgroup $\mathcal{G}$ of $M_n(\complex)$ with a weak $\varepsilon$-approximate fixed point $\xi$ has a common eigenvector $\eta$. However, the techniques  of \cite{KuMaOmRa} cannot yield any estimate of the size of $\varepsilon$ (they cannot even guarantee that $\varepsilon$ will depend on $\frac{1}{n}$ polynomially), nor do they guarantee that the common eigenvector $\eta$ will be polynomially close to the starting weak approximate fixed point $\xi$.

The paper is organized as follows. In Section~\ref{sec:comm} we study finite groups generated by their commutators. We show, in particular, that if every element of the group is a product of a commutator and a scalar, the existence of a weak $\varepsilon$-fixed point for sufficiently small $\varepsilon$ implies reducibility (Theorem~\ref{product with scalar}). In Section \ref{sec:mon} we study monomial groups and prove in Theorem~\ref{monomial group} that a monomial group of unitary matrices having a weak $\varepsilon$-approximate fixed point has a common eigenvector provided that $\varepsilon$ is of the order of magnitude $n^{-11}$. (Recall that all our eigenvectors, common, approximate or otherwise, will be of norm 1.) Some important results on decompositions of unitary groups are presented in Section~\ref{sec:decomp}, and in Section~\ref{sec:gen} we combine the results of the previous sections to give an analogous result for a general finite group of unitaries. In Section~\ref{sec:con} we present the desired result for connected compact groups using a well-known result on connected compact Lie groups. Then we use the fact that the connected component of the identity of a compact group has finite index to extend our finite-group results to arbitrary compact groups of unitaries via a simple device. Finally, in Section~\ref{sec:eigenvec} we present the main result (Theorem~\ref{main}) mentioned above.

We will make frequent use of the properties of the elements of the group as matrices and will consequently prefer to see our groups as subsets of $M_n(\complex)$. We will also use the fact that these matrices act as operators on the underlying space $\complex^n$ which will be equipped with the usual Hilbert space norm, i.e.\ inner-product norm. As a matter of fact, most of our groups will consist of unitary matrices and this general assumption will make our theory work.

%%%%%%%%%%%%%%%%%%%%%%%%%%%%%%%%%%%%%%%
\section{ Commutator groups and fixed points}\label{sec:comm}

{Let us  start with} a well-known fact for which we give a simple proof for the sake of completeness.

\begin{proposition}\label{fixed_point}
Assume that a compact group $\mathcal{G}\subseteq
M_n(\complex)$ of unitaries has an $\varepsilon$-approximate
fixed point $\xi$% with $||\xi||=1$
.  Then $\mathcal{G}$  has  a fixed point $\eta\not=0$ such that $||\xi-\eta|| \leqslant \varepsilon $.
\end{proposition}

\begin{proof}
  Let $\mu$ be the Haar measure on $\mathcal{G}$   and
define
\[
    \eta%=\left(\int_{G\in \mathcal{G}}G\,d\mu\right)\xi
    =\int_{G\in \mathcal{G}}G\xi\,d\mu.
\]
Using standard arguments and the fact that the Haar measure is left invariant, we see that $G\eta=\eta$ for all $G\in \mathcal{G}$, so that $\eta$ is a fixed point. Use the fact that $\mu$ is a positive measure with $\mu(\mathcal{G})=1$ to get
\[
    \|\eta-\xi\| = \sup_{\|\vartheta\|\leqslant1}|\langle\eta-\xi,\vartheta\rangle|= \sup_{\|\vartheta\|\leqslant1}\left|\int_{G\in \mathcal{G}}\langle G\xi-\xi,\vartheta\rangle\,d\mu\right|\leqslant \varepsilon
\]
as desired.
\end{proof}

Observe that the assumption ``of unitaries'' is not necessary in the result above. This is because the existence of the Haar measure is assured solely by the assumption that $\mathcal{G}$ be compact. As a matter of fact we may only assume it is bounded. We can close a bounded group to make it compact, and using the Haar measure we can find a similarity after which all members of the group are unitary operators; we will omit the proof since all this is well known. However, all these assertions are true only up to simultaneous similarity, while one of our main points is to determine the $\varepsilon$ as a polynomial function of $n$. And this point can be studied only if the groups are assumed unitary upfront.

%On the other hand, if we do not assume members of the group to be unitary operators, the definition of the $\varepsilon$-fixed point and its weak version has to be adjusted for the bounds of the group.

We will now show that a weak approximate fixed point may always
be seen as an approximate eigenvector for all elements of the
group. So, let $\mathcal{G}\subseteq M_n(\complex)$ be a
compact group of unitaries and let $\xi\in\complex^n$ of norm
one be a \emph{weak approximate fixed point} of $\mathcal{G}$. We want to show that this vector is an approximate eigenvector of every element of the group with the corresponding   eigenvalue depending on
the element. Introduce for every $G\in\mathcal{G}$ the scalar
\begin{equation}\label{appr eigen vector}
    \lambda_G=\frac{\langle G\xi,\xi \rangle}{|\langle G\xi,\xi \rangle|}.
\end{equation}
Observe that $|\lambda_G|=1$ and $\lambda_{G^*}=\overline{\lambda}_G$ since $\langle G^*\xi,\xi \rangle=\langle \xi,G\xi \rangle = \overline{\langle G\xi,\xi \rangle}$.

\begin{proposition}\label{eigenvector}
Let $n\in\mathbb{N}$ and $\varepsilon>0$ be given. If $\xi\in\complex^n$ %of norm one
is a weak $\varepsilon$-approximate fixed point of a compact group $\mathcal{G}\subseteq M_n(\complex)$ of unitaries, then
\[
    \|G\xi-\lambda_G\xi\|\leqslant\sqrt{2\varepsilon}\ \ \ \hbox{for any}\ \ \ G\in\mathcal{G}.
\]
\end{proposition}

 \begin{proof}
   A straightforward computation reveals
\begin{align*}
    \|G\xi-\lambda_G\xi\|^2 &= \langle G\xi-\lambda_G\xi, G\xi-\lambda_G\xi\rangle \\
    &= \langle G\xi,G\xi\rangle-2\mathbb{R}\mathrm{e}(\langle G\xi,\lambda_G\xi\rangle)+ \langle\lambda_G\xi,\lambda_G\xi\rangle \\
    &= 2-2|\langle G\xi,\xi \rangle|\leqslant2\varepsilon.
\end{align*}
 \end{proof}

\begin{theorem}\label{commutators} Let $n\in\mathbb{N}$ and $\varepsilon>0$ be given. Assume that every element of a group of unitaries $\mathcal{G} \subseteq M_n(\complex)$ is a commutator of two elements of $\mathcal{G}$. If $\xi\in\complex^n$ %of norm one
is a weak $\varepsilon$-approximate fixed point of $\mathcal{G}$, then $\xi$ is also an $\varepsilon'$-approximate fixed point of this group for $\varepsilon'= 4 \sqrt{2\varepsilon}$.
\end{theorem}

Recall that an element of the form $[G,H]=GHG^{-1}H^{-1}$ is called a \emph{commutator} for any $G,H\in\mathcal{G}$; since the group under consideration consists of unitary matrices we may also compute the commutator as $[G,H]=GHG^*H^*$. Of course, not every group of matrices (even if finite) has the property that all of its elements are commutators, the smallest nonabelian counterexample is the group $S_3$, seen as a group of permutation matrices in $M_3(\complex)$. The smallest counterexample of an abstract perfect group (a group that equals its derived group, i.e.\ the group generated by its commutators) in which there are elements that are not commutators, is the group  $G=(C_2 \times C_2 \times C_2 \times C_2) \rtimes A_5$ of size $960$ (cf.\ \cite{GAP}).\\

%\textcolor{red}{Even the derived subgroup may fail to have this property. The smallest example is  $(C_2\times C_2\times Q_8)\rtimes C_3$ where there exists commutators whose product is not a commutator. } However, it has been shown recently (what had been known for some time as Ore's conjecture) that every finite nonabelian simple group has this property \cite{LieOBrShaTie}.

Observe that Theorem \ref{commutators} is a simple corollary of the following lemma.
\begin{lemma}\label{referee}
  Let $G\in\mathcal{G}$ be a commutator and $\xi\in\complex^n$ %of norm one
be a weak $\varepsilon$-approximate fixed point of $\mathcal{G}$. Then
\[
    \|G\xi-\xi\|   \leqslant 4\sqrt{2\varepsilon}.
\]
\end{lemma}

 \begin{proof}
   Write a commutator $G\in\mathcal{G}$ as $G=[A,B]=ABA^*B^*$ for some $A,B\in\mathcal{G}$. We estimate $\|G\xi-\xi\|$ using Proposition \ref{eigenvector}:
\begin{align*}
    \|G\xi-\xi\| &\leqslant\| ABA^*(B^*-\overline{\lambda}_BI_n)\xi\|+\|ABA^*\overline{\lambda}_B\xi-\xi\| \\
    &\leqslant \sqrt{2\varepsilon}+ \| AB\overline{\lambda}_B(A^*-\overline{\lambda}_AI_n)\xi\|+ \|AB\overline{\lambda}_B\overline{\lambda}_A\xi-\xi\| \\
    &\leqslant 2\sqrt{2\varepsilon}+ \| A\overline{\lambda}_B\overline{\lambda}_A(B-\lambda_BI_n)\xi\|+ \|A\overline{\lambda}_A\xi-\xi\| \leqslant 4\sqrt{2\varepsilon}.
\end{align*}
 \end{proof}

The following corollary is a consequence of Theorem \ref{commutators} and Proposition \ref{fixed_point}.

\begin{corollary}\label{bla}
Given $n\in\mathbb{N}$, let every element of a compact group of unitaries $\mathcal{G}\subseteq M_n(\complex)$ be a commutator of two elements of $\mathcal{G}$. Let $0<\varepsilon<\dfrac{1}{32}$. If the group has a weak $\varepsilon$-approximate fixed point $\xi$, %$||\xi||=1$,
then it has a fixed point $\eta\not=0$ satisfying $||\eta-\xi|| \leqslant 4 \sqrt{2\varepsilon}$.\\
\end{corollary}

%In the proof of the following lemma and in the sequel we denote by $\Omega_n$ the group of $n$th roots of unity. Also, let $\mathbb{T}=\{z\in \mathbb{C}: |z|=1\}$.

%\begin{lemma}\label{morphism}
  %Let $n\in\mathbb{N}$ and $\varepsilon>0$ be given, and assume that $\varepsilon<\displaystyle\frac{1}{32 n^2}$. Let $\mathcal{G}\subseteq M_n(\complex)$ be a group of unitary matrices. Then there exists a unique group homomorphism $\rho:\mathcal{G}\rightarrow \mathbb{T}$ such that
%\end{lemma}

%\begin{proof}
  %Assume with no loss of generality that $\mathcal{G}=\overline{\mathbb{T}\mathcal{G}}$, where , and ``bar" denotes the closue in the Euclidean topology.  Now note that it is sufficient to prove the result for $\mathcal{G}_0=\{G\in\mathcal{G}: \det(G)=1\}$.  This is because $\mathcal{G}_0$ also satisfies the assumptions of the theorem and because every element of $\mathcal{G}$ is an $\mathbb{T}$-multiple of an element of $\mathcal{G}_0$.
%\end{proof}

\begin{theorem}\label{product with scalar} Let $n\in\mathbb{N}$ and $\varepsilon>0$ be given, and assume that $n\geqslant 2$ and that $\varepsilon<\displaystyle\frac{1}{32 n^2}$.
Let $\mathcal{G}\subseteq M_n(\complex)$ be a group of unitary matrices
such that every element of $\mathcal{G}$ is a product of a commutator of two elements of $\mathcal{G}$ and a scalar.

If $\mathcal{G}$ has a weak $\varepsilon$-approximate fixed point $\xi\in\complex^n$, %$||\xi||=1$,
then $\mathcal{G}$ has a common eigenvector $\eta$, such that $||\xi-\eta||\leqslant 4\sqrt{2\varepsilon} <\displaystyle\frac{1}{n}$.
\end{theorem}

  \begin{proof}
    Assume with no loss of generality that $\mathcal{G}=\overline{ \mathbb{T}\cdot \mathcal{G}}$, where $\mathbb{T}=\{z\in\complex\,|\ |z|= 1\}$, and ``bar" denotes the closure in the Euclidean topology.  Now note that it is sufficient to prove the result for $\mathcal{G}_0= \{G\in \mathcal{G}\,|\, \det(G)=1\}$.  This is because $\mathcal{G}_0$ also satisfies the assumptions of the theorem and because every element of $\mathcal{G}$ is a $\mathbb{T}$-multiple of an element of $\mathcal{G}_0$.

 Choose a $G\in\mathcal{G}_0$ and write
 \begin{equation}\label{eq:1}
G=\lambda[A,B]
\end{equation} for some $A,B\in\mathcal{G}_0$ and a scalar $\lambda$. Note that $\lambda I=G[A,B]^{-1}$, so $\lambda I\in\mathcal{G}_0$ and hence $\lambda^n=1$.

By Lemma \ref{referee} we have that
\[
    \|[A,B]\xi-\xi\|\leqslant 4\sqrt{2\varepsilon}
\]
implying that
\begin{equation}\label{eq:estimate}
    \|G\xi-\lambda\xi\|=\|\lambda[A,B]\xi-\lambda\xi\|\leqslant 4\sqrt{2\varepsilon}.
\end{equation}

Since the adjacent $n$-th roots of unity are at the distance
\[d_n=\left|e^{2\pi i/n}-1\right|\geqslant\frac{4}{n},\]
and
\[4\sqrt{2\varepsilon}<4\sqrt{2\frac{1}{32n^2}}=\frac{1}{n}<\frac{1}{2} d_n,\]
we conclude that $\lambda$ defined by \eqref{eq:1}    is
unique. Consequently, we may define a map $\varrho$ going from
$\mathcal{G}$ to $\mathcal{G}$, defined by $\varrho\colon
G\mapsto\lambda I$. We will show that this map is a group
homomorphism.  Indeed, choose $G=\lambda[A,B]$ for some
$A,B\in\mathcal{G}$ and a scalar $\lambda=\varrho(G)$ and
choose $\widetilde{G} = \widetilde{\lambda}\,[\widetilde{A},
\widetilde{B}]$ for some
$\widetilde{A},\widetilde{B}\in\mathcal{G}$ and a scalar
$\widetilde{\lambda}=\rho(\widetilde{G})$. Then,
\begin{align*}
    \|G\widetilde{G}\xi-\lambda\widetilde{\lambda}\xi\|& =\|(\lambda[A,B]\widetilde{\lambda}[\widetilde{A},\widetilde{B}]\xi-\lambda\widetilde{\lambda}\xi)\| =\|\lambda\widetilde{\lambda}([A,B][\widetilde{A},\widetilde{B}]\xi-\xi)\|  \\
    &\leqslant\|[A,B]([\widetilde{A},\widetilde{B}]\xi-\xi)\|+\|[A,B]\xi-\xi\|\leqslant8\sqrt{2\varepsilon}<\frac{2}{n}
\end{align*}
Hence
\begin{align*}
\left|\varrho(G\widetilde{G})-\varrho(G)\varrho(\widetilde{G})\right|
&= \left|\left| \left(\varrho(G\widetilde{G})-\varrho(G)\varrho(\widetilde{G})\right)\xi\right|\right|\\
&\leqslant \left|\left|\varrho(G\widetilde{G})\xi - (G\widetilde{G})\xi\right|\right|+\left|\left|(G\widetilde{G})\xi-\varrho(G)\varrho(\widetilde{G})\xi\right|\right|\\
&<\frac{1}{n}+\frac{2}{n}<d_n.
\end{align*}
Thus $\varrho(G)\varrho(\widetilde{G})=\lambda\widetilde{\lambda}= \varrho(G\widetilde{G})$.  As in the proof of Proposition \ref{weak appr fixed} let $\mu$ be the Haar measure on $\mathcal{G}$ and introduce
\[
    \eta= \int_{G\in \mathcal{G}}\overline{\varrho(G)}G\xi\,d\mu,
\]
so that
\[
    H\eta=\int_{G\in \mathcal{G}}\overline{\varrho(G)}HG\xi\,d\mu
    =\varrho(H) \int_{G\in \mathcal{G}} \overline{\varrho(HG)}HG\xi\,d\mu =\varrho(H)\eta
\]
for all $H\in \mathcal{G}$, implying that $\eta$ is a common eigenvector of all elements of the group. Also
\[
    ||\eta-\xi||=\sup_{\|\vartheta\|\leqslant1}|\langle\eta-\xi| \vartheta\rangle|=\sup_{\|\vartheta\|\leqslant1}\left|\int_{G\in \mathcal{G}}\overline{\varrho(G)}\langle G\xi-\varrho(G)\xi|
        \vartheta\rangle\,d\mu\right| \leqslant4\sqrt{2\varepsilon}
\]
as desired.
  \end{proof}   

%%%%%%%%%%%%%%%%%%%%%%%%%%%%%%%

\section{Monomial groups}\label{sec:mon}

In the following proposition we need a version of the
well-known Hardy-Littlewood-P\'olya theorem \cite{HarLitPolI},
given as Theorem 368 on p.\ 261 of \cite{HarLitPolII} (the
inequality is often called the ``Rearrangement Inequality").
Let $x_1\geqslant x_2\geqslant \cdots\geqslant x_n\geqslant0$
and $y_1 \geqslant y_2\geqslant \cdots\geqslant y_n\geqslant0$
be two  $n$-tuples  of reals and let $\pi$ be a permutation of
$n$ indices. Then,
\begin{equation}\label{HLP}
    \sum_{i=1}^nx_iy_{\pi(i)}\leqslant  \sum_{i=1}^nx_iy_{i}.
\end{equation}

In what follows we will denote by $|\xi|$ the vector whose entries consist of the absolute values of the corresponding entries of $\xi$. Similarly, $|G|$ will denote the matrix whose entries are equal to the absolute values of the corresponding entries of $G\in M_n(\complex)$. For a group $\mathcal{G}\subseteq M_n(\complex)$ we will denote by $|\mathcal{G}|$ the set of all $|G|$ for $G\in\mathcal{G}$. Recall that a unitary group $\mathcal{G}$ is $\emph{monomial}$ whenever $|\mathcal{G}|$ consists of permutation matrices which implies that it is a permutation group. Furthermore, recall that a group of matrices is called \emph{indecomposable} if it has no nontrivial invariant subspace spanned by a subset of the standard basis vectors.

%\begin{lemma}\label{distance} If $\alpha$ and $\beta$ are two distinct complex numbers of modulus one, then there exists such a $k\in\mathbb{N}$ that $|\alpha^k+\beta^k| \leqslant1$. \end{lemma}

% \prf Assume with no loss that $\beta=1$ (since otherwise we can divide the expression inside the modulus sign by $\beta^k$). If $\alpha^m=1$ for some integer $m$, then the powers of $\alpha$ divide the unit circle into equal arcs. With $m$ even, $|\alpha^{m/2}+1|=0$ and we are done. With $m$ odd, $-1$ belongs to one of those arcs so the distance between $-1$ and powers of $\alpha$ is shorter than the distance between consecutive arc points, which is largest when $\alpha=e^{2\pi i/3}$. So, the desired distance is the greatest at $m=3$ where it is equal to $1$. If $\alpha^m\ne 1$ for any integer $m$, then powers of $\alpha$ come arbitrarily close to $-1$. \qed

In the following proposition we will need the vector
\begin{equation}\label{fixed}
    \eta=\frac{1}{\sqrt{n}}\left(
    \begin{array}{c}
      1 \\
      1 \\
      \vdots \\
      1
    \end{array}
    \right).
\end{equation}

\begin{proposition}\label{indecomposable}
Let $n\in\mathbb{N}$ and $\varepsilon>0$ be given. If an indecomposable monomial unitary group $\mathcal{G}\subseteq M_n(\complex)$ has a weak $\varepsilon$-approximate fixed point $\xi$, then:
\begin{description}
  \item[(a)] $|\mathcal{G}|$ has a fixed point $\eta$ such that $\|\,\eta-|\xi|\,\|\leqslant\varepsilon_1$, where $\varepsilon_1=n \sqrt{n\varepsilon}$;
  \item[(b)] if $\varepsilon_1<\dfrac{1}{3}$ and $\xi=|\xi|$ then $\eta$ as defined by \eqref{fixed} is a weak $\varepsilon_2$-approximate fixed point of $\mathcal{G}$, where $\varepsilon_2 =3\varepsilon_1$;
  \item[(c)] if $\varepsilon_1<\dfrac{1}{3}$ then there is a weak $\varepsilon_2$-approximate fixed point $\zeta$ of $\mathcal{G}$, where $\varepsilon_2 =3\varepsilon_1$, and $|\zeta|=\eta$ as defined by \eqref{fixed}.
  %\item[(d)] if the assumptions of the proposition are satisfied for all $\varepsilon>0$ then $\eta$ is an eigenvector of all members of $G$;
  %\item[(e)] if the assumptions of the proposition are satisfied for all $\varepsilon>0$ then every member of $\mathcal{G}$ is a multiple of a scalar and a permutation matrix.
\end{description}
\end{proposition}

 \begin{proof}
   Assume with no loss that $\xi=|\xi|$ (otherwise use an
appropriate unitary diagonal similarity on the group). Then,
assume with no loss that $\xi_1\geqslant \xi_2\geqslant
\cdots\geqslant \xi_n\geqslant0$ (otherwise use an appropriate
permutational similarity which is necessarily unitary).  Now,
any $G\in\mathcal{G}$ is determined by a permutation $\pi$ of
$n$ indices and a choice of $n$ complex numbers $\gamma_1,
\gamma_2, \ldots, \gamma_n$ of modulus one, so that for any $x
\in\complex^n$ we have $(Gx)_i=\gamma_ix_{\pi(i)}$. By the
definition of a weak approximate fixed point and \eqref{fixed} we get
\[
    1-\varepsilon\leqslant |\langle G\xi,\xi\rangle|\leqslant \left| \sum_{i=1}^n \gamma_i\xi_{\pi(i)}\xi_i \right| \leqslant 1.
\]
The fact that $\mathcal{G}$ is indecomposable implies that the permutation group corresponding to $|\mathcal{G}|$ acts transitively on the standard basis vectors. So, for any $k,1\leqslant k < n$, there exists a permutation $\pi$ corresponding to a certain $|G|\in|\mathcal{G}|$ such that $\pi(k)= k+1$, implying
\[
    1-\varepsilon\leqslant \xi_k\xi_{k+1}+\sum_{\overset{i=1}{i\neq k}}^n \xi_i\xi_{\pi(i)}\leqslant 2\xi_k\xi_{k+1} +\sum_{\overset{i=1}{i\neq k,k+1}}^n \xi_i^2.
\]
In the first sum above we use (\ref{HLP}) with $x_i$'s equal to $\xi_i$'s after omitting $\xi_k$ and $y_i$'s equal to $\xi_i$'s after omitting $\xi_{k+1}$ in order to get the estimate. Recall that $\sum_{i=1}^n\xi_i^2=1$ to get, after rearranging,
\[
    \xi_k^2+\xi_{k+1}^2-2\xi_k\xi_{k+1}\leqslant\varepsilon,
\]
so that $0\leqslant \xi_k-\xi_{k+1}\leqslant\sqrt{\varepsilon}$. It follows that  $0\leqslant \xi_1-\xi_{n}= \sum_{k=1}^{n-1}(\xi_k-\xi_{k+1}) \leqslant(n-1)\sqrt{\varepsilon}$, implying finally
\begin{equation}\label{max}
    \max_{1\leqslant i,j\leqslant n}|\xi_i-\xi_j|\leqslant(n-1)\sqrt{\varepsilon}.
\end{equation}
We now want to show that the vector $\xi$ is close to a fixed point of the group $|\mathcal{G}|$. Actually, we may choose for that purpose the vector $\eta$ defined by \eqref{fixed}. The simple observation $n\xi_1^2 \geqslant \sum_{i=1}^n \xi_i^2=1\geqslant n\xi_n^2$ yields $\xi_1\geqslant \eta_i\geqslant\xi_n$ for all $i, 1\leqslant i\leqslant n$, so that \eqref{max} implies $\displaystyle |\eta_i-\xi_i|\leqslant\max_{1\leqslant
j\leqslant n}|\xi_i-\xi_j|\leqslant(n-1)\sqrt{\varepsilon}$,
for all $i, 1\leqslant i\leqslant n$, and consequently
\[
    \|\eta-\xi\|\leqslant (n-1)\sqrt{n\varepsilon}<n\sqrt{n\varepsilon}=\varepsilon_1
\]
yielding \textbf{(a)}. This expression also implies that
\[
    \Bigl||\langle G\xi,\eta\rangle|-|\langle G\xi,\xi\rangle|\Bigr|\leqslant |\langle G\xi,\eta-\xi\rangle| \leqslant \varepsilon_1
\]
and similarly $\Bigl||\langle G\eta,\eta\rangle|-|\langle
G\xi,\eta\rangle|\Bigr| \leqslant \varepsilon_1$. So,
\[
    \Bigl||\langle G\xi,\xi\rangle|-|\langle G\eta,\eta\rangle|\Bigr|\leqslant 2 \varepsilon_1,
\]
and the fact that $\varepsilon\leqslant\varepsilon_1$ now implies that
\[
    |\langle G\eta,\eta\rangle|\geqslant1-\varepsilon_2,
\]
where $\varepsilon_2=3\varepsilon_1$, thus showing \textbf{(b)}. (Note that $\xi=|\xi|$ was assumed ``with no loss'' at the beginning of this proof.) \textbf{(c)} The general case is now obtained using a unitary diagonal similarity on the group. If this similarity is given by a diagonal $U$, then $\zeta=U\eta$ yields the desired result.
 \end{proof}

%Now, observe that the weak approximate fixed point $\eta$ is independent of the choice of $\varepsilon$, so we can let $\varepsilon$ in the definition \eqref{weak appr fixed} go to zero, and get $|\langle G\eta,\eta\rangle|\geqslant1$. Since the inverse inequality is obvious, we see that $\lambda_G=\langle G\eta,\eta\rangle$ is of modulus one and is by Proposition \ref{eigenvector} the eigenvalue of $G$ corresponding to the eigenvector $\eta$; we have thus shown \textbf{(d)}. It follows easily that $G$ is a product of $\lambda_G$ with a permutation matrix thus proving also \textbf{(e)}.
%It remains to see that for $\varepsilon$ and therefore $\varepsilon_1$ small enough the diagonal part of $\mathcal{G}$ is scalar. Assume towards contradiction that there is a diagonal element $G=\mathrm{Diag}(\gamma_1,\gamma_2,\ldots, \gamma_n)$ of $\mathcal{G}$ which is not scalar. So, $\gamma_i$'s are complex numbers of modulus one and there are indices $j$ and $l$, say, such that $1\leqslant j<l\leqslant n$ and that $\gamma_j\neq\gamma_l$. It follows that     \[       1-\varepsilon_1 \leqslant |\langle G\eta,\eta\rangle|=\frac{1}{n}\left|\gamma_j^k+\gamma_l^k +\sum_{\overset{i=1}{i\neq j,k}}^n \gamma_i^k \right|\leqslant\frac{|\gamma_j^k+\gamma_l^k|+n-2}{n} \leqslant\frac{n-1}{n}     \]    by Lemma \ref{distance}. This estimate clearly contradicts a small enough choice of $\varepsilon$ and consequently of $\varepsilon_1$.

\begin{lemma}\label{observation}
Let $\varphi\in [-\pi,\pi]$.  Then
\begin{equation}\label{eqobservation}
    \frac{1}{2} |\varphi|\leqslant \frac{2}{\pi} |\varphi|\leqslant |e^{\varphi i}-1|\leqslant |\varphi|.
\end{equation}
\end{lemma}

  \begin{proof}
    Clearly, it suffices to give the proof only for
$\varphi\geqslant0$ since $|e^{-\varphi i}-1|=|e^{\varphi
i}-1|$. Observe that the function $\displaystyle
f(\varphi)=2\sin\left(\frac{\varphi}{2}\right)$ is
 convex upwards  on $[-\pi,\pi]$ since its second
derivative is nonpositive on this interval.
This implies that the graph of the curve lies above the line segment connecting any two points on the graph and, in particular, the endpoints of this interval.  So, $\displaystyle \frac{2}{\pi}\varphi\leqslant f(\varphi)= 2 \sin\left( \frac{\varphi}{2}\right)$ since $\varphi \geqslant0$, %(and a similar inequality may be obtained for $\varphi<0$).
and the second of the three inequalities in \eqref{eqobservation} follows. Now, the first one is obvious and the third follows from the fact that a line segment is shorter than the arc connecting the same points.
  \end{proof}

\begin{lemma}\label{observationA}
Let $\varphi_j\in [-\pi,\pi]$ for $j=1,2,\ldots,k$.  If
\begin{equation}\label{assumption}
    \left|1+\sum_{j=1}^ke^{\varphi_ji}\right|\ge (k+1)(1-\varepsilon);\qquad  (0<\varepsilon <1)
\end{equation}
then
\begin{equation}\label{eqobservationA}
     \sum_{j=1}^k|\varphi_j|%\leqslant k-\sum_{j=1}^kx_j
     <\pi\sqrt{k}(k+1)\sqrt{\frac{ \varepsilon}{2}}.
\end{equation}
\end{lemma}

  \begin{proof}
    Let $e^{\varphi_j i}=x_j+iy_j, x_j,y_j\in\reals$ for $j=1,2,\ldots,k$.
We first want to show that \eqref{assumption} implies
\begin{equation}\label{eqobservationA1}
     \frac{2}{\pi^2}\sum_{j=1}^k|\varphi_j|^2\leqslant k-\sum_{j=1}^kx_j.
\end{equation}
To this end observe that for any $j\in\{1,2,\ldots,k\}$ and $x_j'=1-x_j$ we have $\displaystyle 2x_j'=x_j'^2+y_j^2 =|e^{\varphi_j i}-1|^2\geqslant \frac{4}{\pi^2}|\varphi_j|^2$ by Lemma \ref{observation}. We sum these estimates and divide by 2 to get \eqref{eqobservationA1}.

Clearly
\begin{equation*}
  \begin{split}
    \left|1+\sum_{j=1}^ke^{\varphi_ji}\right|^2=
    \left(1+\sum_{j=1}^kx_j\right)^2+\left(\sum_{j=1}^ky_j\right)^2&=\\
    \left(1+2\sum_{j=1}^kx_j\right)+\sum_{j=1}^k\sum_{l=1}^k(x_jx_l+y_jy_l).
\end{split}
\end{equation*}
Next, observe that $x_jx_l+y_jy_l=\mathrm{Re}\left(e^{\varphi_ji}\overline{e^{\varphi_li}}\right)\leqslant1$ for all $j,l= 1,2,\ldots,k$, so that
\[
    \left|1+\sum_{j=1}^ke^{\varphi_ji}\right|^2-(k+1)^2\leqslant
    2\left(\sum_{j=1}^kx_j-k\right).
\]
Combine this estimate with \eqref{assumption} and \eqref{eqobservationA1} to get
\begin{equation}\label{eqobservationA4}
    \frac{4}{\pi^2}\sum_{j=1}^k|\varphi_j|^2\leqslant (k+1)^22\varepsilon,
\end{equation}
where we have also used the simple estimate $1-(1-\varepsilon)^2\leqslant2\varepsilon$. So, finally
\[
    \sum_{j=1}^k|\varphi_j|\leqslant \left(k\sum_{j=1}^k|\varphi_j|^2\right)^{1/2} <
    \pi\sqrt{k}(k+1)\sqrt{\frac{ \varepsilon}{2}} %= \frac{n\pi}{2}\sqrt{n\varepsilon}\le \pi/2<\pi.
\]
and the desired Inequality \eqref{eqobservationA} follows.
  \end{proof}

In the following proposition we need the standard notations $\displaystyle ||\xi||_1=\sum_{i=1}^n|\xi_i|$ and  $\displaystyle ||\xi||_2^2=\sum_{i=1}^n|\xi_i|^2$ for any $\xi\in\complex^n$.

\begin{proposition}\label{prop:alpha}
Let $\displaystyle 0<\varepsilon<\frac{2}{n^3}$
 and let $g_1,\ldots, g_n\in \mathbb{T}$ be
such that
$$g_1g_2\cdots g_n=1,\ \mbox{ and}$$
$$|g_1+\cdots+g_n|\ge n(1-\varepsilon).$$
Then there exists $\alpha\in \mathbb{T}$ such that $\alpha^n=1$, and
\begin{eqnarray*}
||(g_1,\ldots, g_n)-\alpha (1,\ldots, 1)||_1 &<& \pi n\sqrt{2n\varepsilon}\ \ \mbox{\rm and} \\
||(g_1,\ldots, g_n)-\alpha (1,\ldots, 1)||_2 &<& \pi n\sqrt{2\varepsilon}.\\
\end{eqnarray*}
\end{proposition}

  \begin{proof}
    Write $k=n-1$, and let $\varphi_j\in (-\pi,\pi]$ be such that
\[
    \frac{g_{j+1}}{g_1}=e^{\varphi_j i}
\]
for $j=1,\ldots, k$. Moreover, let $\alpha\in \mathbb{T}$ and $\varphi\in (-\pi/n,\pi/n]$ be such that $\alpha^n=1$ and $g_1=\alpha e^{\varphi i}$. By Lemma \ref{observationA} we have that
\begin{equation}\label{eqobservationA2}
    \sum_{j=1}^k|\varphi_j|<\pi\sqrt{k}(k+1)\sqrt{\frac{ \varepsilon}{2}}<\pi\sqrt{\frac{n^3\varepsilon}{2}} <\pi.
\end{equation}
From $g_1\cdots g_n=1$ we then get that
$n\varphi+\varphi_1+\cdots+\varphi_k=0$ (since the inequalities
in  \eqref{eqobservationA2} and the fact that
$n|\varphi|\leqslant\pi$ give $-2\pi<n\varphi +\varphi_1
+\cdots+\varphi_k<2\pi$).  So, by \eqref{eqobservationA2}
again, we get
\begin{equation}\label{eqobservationA3}
    n|\varphi| \leqslant \sum_{j=1}^k|\varphi_j|<\pi\sqrt{\frac{ n^3\varepsilon}{2}}.
\end{equation}
Hence we have
\[
    ||(g_1,\ldots, g_n)-\alpha(1,1,\ldots, 1)||_1 \leqslant |\varphi|+\sum_{j=1}^k|\varphi+\varphi_j|\leqslant n|\varphi|+\sum_{j=1}^k|\varphi_j|<\pi\sqrt{2n^3\varepsilon},
\]
where we have first used \eqref{eqobservation}, and then \eqref{eqobservationA2} and \eqref{eqobservationA3}.

The inequality involving the 2-norm is verified as follows:
\begin{align*}
    ||(g_1,\ldots, g_n)-\alpha(1,1,\ldots, 1)||_2^2 &\leqslant |\varphi|^2+\sum_{j=1}^k|\varphi+\varphi_j|^2 \leqslant \\ |\varphi|^2 +2\sum_{j=1}^k(|\varphi|^2+|\varphi_j|^2)&=(2n-1)|\varphi|^2+2\sum_{j=1}^k|\varphi_j|^2 <2\pi^2n^2\varepsilon.
\end{align*}
using \eqref{eqobservationA4} (recalling that $k+1=n$) and \eqref{eqobservationA3}.
  \end{proof}

\begin{theorem}\label{reducibility}
Let $n\in\mathbb{N}$ and $\varepsilon>0$ be given. If $n\geqslant 2$ and $\displaystyle \varepsilon<\frac{1}{3600 n^{11}}$ then every indecomposable monomial unitary group $\mathcal{G} \subseteq M_n(\complex)$ with a weak $\varepsilon$-approximate fixed point $\xi$ has a common eigenvector $\zeta$ such that $||\xi-\zeta|| %\leqslant \varepsilon'=\pi\sqrt{6}\sqrt[4]{n^7\varepsilon}
<\displaystyle\frac{1}{n}$.
\end{theorem}

\begin{proof}
     Assume first with no loss that $\xi=|\xi|$. We then use Proposition \ref{indecomposable}\textbf{(b)} to see that $\eta$ defined by \eqref{fixed} is a weak $\widetilde{\varepsilon}$-approximate fixed point for
\[
 \widetilde{\varepsilon}= 3n\sqrt{n\varepsilon}< \frac{3}{60}\sqrt{\frac{ n^3}{n^{11}}}= \frac{1}{20 n^4}.
\]
We assume, with no loss, that $\mathcal{G}$ is compact and $\mathbb{T}$-homogeneous, i.e., that $\mathcal{G}= \overline{\mathbb{T}\mathcal{G}}$.

Let $\mathcal{G}_0$ be the subgroup of $\mathcal{G}$ consisting of those elements of $\mathcal{G}$ whose weights multiply to $1$, i.e., $\mathcal{G}_0$ is the kernel of the group homomorphism $G\mapsto \det(G)\mathrm{sign}(|G|)$ (here we abuse the notation and use $\mathrm{sign}(|G|)\in\{\pm 1\}$ to denote the parity of the permutation associated to $|G|$). One can think of this homomorphism as assigning to every matrix of $\mathcal{G}$ the product of its non-zero entries. Note that $\mathcal{G}_0$ is large in the sense that every element of $\mathcal{G}$ is a scalar multiple of an element of $\mathcal{G}_0$.

Next, denote by $\Omega_n$ the group of all complex $n$-th roots of unity. We want to prove the existence and uniqueness of a mapping $\alpha\colon \mathcal{G}_0\to\Omega_n$ such that for all $G\in\mathcal{G}_0$ we have that
\[
    \|G\eta-\alpha(G)\eta\|<
    \varepsilon'=\pi n \sqrt{2\widetilde{\varepsilon}}.
\]
Indeed, for any such $G$ there exist $g_1,\ldots, g_n\in \mathbb{T}$ such that $g_1g_2\cdots g_n=1$, and $G=\mathrm{Diag}(g_1,\ldots, g_n)\,|G|$ by the construction above. By Proposition \ref{indecomposable}\textbf{(b)} we have %The fact that $|G|\eta=\eta$ shows that
\[
    |\langle G\eta,\eta\rangle|\geqslant1-\widetilde{\varepsilon},
\]
which implies that $|g_1+\ldots+ g_n|\geqslant n(1- \widetilde{\varepsilon})$. By Proposition \ref{prop:alpha} there exists an $\alpha(G)\in \Omega_n$ such that
\[
    \|G\eta-\alpha(G)\eta\|<\varepsilon'< \frac{\pi}{\sqrt{10}n} < \frac{1}{n}. %\sqrt{\frac{6\pi^2n^2 \sqrt{A}}{(n)^4}}.
\]
Now, if $d_n$ is the distance between the adjacent $n$-th roots of unity, then, {by Lemma \ref{eqobservation}, $d_n \geqslant 4/n$.} Then $\alpha$ is unique: Indeed, if there were a $\beta(G)\in \Omega_n$ with the same property, we would get
\begin{equation*}
  \begin{split}
   |\alpha(G)-\beta(G)|=\|\alpha(G)\eta-\beta(G)\eta\| & \leqslant \|G\eta-\beta(G)\eta\| +\|G\eta-\alpha(G)\eta\| \\
     & <2\varepsilon'<d_n
\end{split}
\end{equation*}
yielding $\alpha(G)=\beta(G)$. But this is true for every $G$.

Next we want to show that $\alpha$ is a group homomorphism% as soon as $\displaystyle \varepsilon' <\frac{3}{n}$
. Clearly
\[
  \|GH\eta-\alpha(G)\alpha(H)\eta\|\leqslant \|G(H\eta-\alpha(H)\eta)\| + \|\alpha(H)(G\eta-\alpha(G)\eta)\|<2\varepsilon'
\]
and since also $\|GH\eta-\alpha(GH)\eta\|\leqslant\varepsilon'$, we get
\[
    |\alpha(GH)-\alpha(G)\alpha(H)|<3\varepsilon'<d_n
\]
so that $\alpha(GH)=\alpha(G)\alpha(H)$ as desired.

Let $\mathcal{G}_1$ be the kernel of $\alpha$ and note that
$\eta$ is an $\varepsilon'$-approximate fixed point for $\mathcal{G}_1$.  Hence by Proposition \ref{fixed_point} there is a nonzero fixed point $\zeta$, $||\zeta-\eta||<\varepsilon'$, for $\mathcal{G}_1$. By considerations above, respectively by Proposition \ref{indecomposable}\textbf{(a)}, we have
\[
    ||\zeta-\eta|| \leqslant \varepsilon'=\pi\sqrt{6} \sqrt[4]{n^7\varepsilon}\leqslant\frac{\pi}{n\sqrt{10}},\ \ \mbox{respectively}\ \ ||\eta-\xi|| \leqslant n\sqrt{n \varepsilon} \leqslant\frac{1}{60n^4}.
\]
Since $\alpha$ is $\Omega_n$-homogeneous we have that every element of $\mathcal{G}_0$ and hence also every element of $\mathcal{G}$ is a scalar multiple of an element of $\mathcal{G}_1$.  Hence $\zeta$ is a common eigenvector for $\mathcal{G}$. The above two estimates combined give the desired inequality $||\xi-\zeta||<\dfrac{1}{n}$.
 \end{proof}

%%%%%%%%%%%%%%%%%%%%%%%%%%%%%%%%%

\section{Block decompositions}\label{sec:decomp} %%%popravki tu

Let $\mathcal{G}$ be a group of unitary $n\times n$ matrices with a weak $\varepsilon$-approximate point $\xi$. Assume that $\mathcal{G}$ is either block diagonal or block monomial with respect to orthogonal decomposition (for {the} definition of a block-monomial group {we refer the reader to} the first paragraph of \cite[Section 2.2]{MasRad})
\[
    \mathcal{V}=\mathcal{V}_1\oplus \mathcal{V}_2 \oplus\ldots\oplus \mathcal{V}_k.
\]
Here, the spaces $\mathcal{V}_i$ may be of different dimensions. We do not assume that $\mathcal{G}$ is irreducible. In particular, if $\mathcal{G}$ is block-diagonal, it can only be irreducible if $k=1$.
%%%in particular if the group $\mathcal{G}$ is irreducible, then $k=1$ and $\mathcal{V}_1=\mathcal{V}$.
For $i=1,\ldots, k$ we write $n_i= \mathrm{dim}(\mathcal{V}_i)$, and we let $\xi_i$ be the $i$-th component of $\xi$ with respect to this decomposition.
%%% $G_i= G|_{\mathcal{V}_i}$ for $G \in\mathcal{G}$, and $\mathcal{G}_i=\{G_i ; G\in\mathcal{G}\}$.
Furthermore, introduce
$$
\mathcal{G}_{ii}=\left\{G|_{\mathcal{V}_i}\,|\, G\in\mathcal{G}, G(\mathcal{V}_i)\subseteq \mathcal{V}_i\right\}%\subseteq L(\mathcal{V}_i,\mathcal{V}_i)
\subseteq M_{n_i}(\complex);
$$
we view the groups $\mathcal{G}_{ii}$ either as diagonal blocks of those members of $\mathcal{G}$ for which $\mathcal{V}_i$ is invariant or as groups of linear mappings from $\mathcal{V}_i$ to itself. Clearly, for some $i\in\{1, \ldots, k\}$ we have that $\xi_i\not=0$, since the sum of the squares of their norms equals 1.

The following lemma is useful in both situations described above.

\begin{lemma}\label{eigen}
%%% dodal stavek
%Assume the setup and notation from above.  Then the following assertions hold.
  \begin{description}
    \item[(a)] If $\xi_i\not=0$ for some $i\in\{1, \ldots, k\}$ and $\varepsilon_i= \displaystyle \frac{\varepsilon} {||\xi_i||^2}$, then $\displaystyle \widetilde{\xi}_i= \frac{\xi_i}{||\xi_i||}$ is a weak $\varepsilon_i$-approximate fixed point for $\mathcal{G}_{ii}$.
    %%%% morda izpustimo n_1+...n_k=n v opisu spodaj (ker to sledi iz zgornjega)
    \item[(b)] Let $a_1,\ldots, a_k$ be non-negative real numbers such that $a_1+\ldots+a_k=1$.  Then there exists an $i\in\{1,\ldots, k\}$ such that $a_i\ge (n_i/n)$.
    \item[(c)] There always exists an $i\in\{1, \ldots, k\}$ such that $\displaystyle \widetilde{\xi}_i= \frac{\xi_i} {||\xi_i||} \neq0$ is a weak $\widetilde{\varepsilon}$-approximate fixed point for $\mathcal{G}_{ii}$, where $\displaystyle \widetilde{\varepsilon} = \frac{n}{n_i}\varepsilon$.
  \end{description}

\end{lemma}

\begin{proof}
  \textbf{(a)} %%% posplosil dokaz za blocno-diagonalni G
  Let $G\in\mathcal{G}$ be such that $G(\mathcal{V}_i)\subseteq \mathcal{V}_i$ and let $H=G|_{\mathcal{V}_i}$.  For $j\not=i$, let $\pi_j\not=i$ be the unique index such that $G(\mathcal{V}_j)\subseteq \mathcal{V}_{\pi_j}$.
    Then
  \begin{equation*}
    \begin{split}
     (1-\varepsilon) & \leqslant |\langle  G\xi,\xi\rangle| \leqslant \left|\langle H\xi_i,\xi_i\rangle + \sum_{j\not=i} \langle G\xi_j,\xi_{\pi_j}\rangle \right|\\
     &\leqslant |\langle H\xi_i, \xi_i\rangle|+\sum_{j\not=i} ||\xi_j||\cdot ||\xi_{\pi_j}|| \\
     &\leqslant |\langle H\xi_i, \xi_i\rangle|+\sum_{j\not=i} ||\xi_j||^2\\
       &= |\langle H\xi_i, \xi_i\rangle|+1-||\xi_i||^2
  \end{split}
  \end{equation*}
  (we used the Rearrangement Inequality mentioned at the beginning of Section 3 to go from the second to the third line).
  Hence $|\langle H\xi_i, \xi_i\rangle|\geqslant ||\xi_i||^2-\varepsilon$ and therefore
  \[
  |\langle H\widetilde{\xi_i}, \widetilde{\xi}_i\rangle|\geqslant 1-\frac{\varepsilon}{||\xi_i||^2} = 1-\varepsilon_i.
  \]
  \textbf{(b)} Suppose, if possible, that for every $i$ we have $a_i<n_i/n$.  Then $a_1+\ldots+a_n<(n_1/n)+\ldots+(n_k/n)=1$.  A contradiction.\\ \textbf{(c)} Immediate consequence of \textbf{(a)} and \textbf{(b)}.
\end{proof}

\begin{theorem}\label{monomial group}
Let $n\in\mathbb{N}$ and $\varepsilon>0$ be given. If $n\geqslant 2$ and $\displaystyle\varepsilon<\frac{1}{3600 n^{11}}$, then every monomial unitary group $\mathcal{G} \subseteq M_n(\complex)$ with a weak $\varepsilon$-approximate fixed point $\xi$ %, $||\xi||=1$,
has a common eigenvector $\zeta$.
\end{theorem}

 \begin{proof}
If $\mathcal{G}\subseteq M_n(\complex)$ is indecomposable, we are done by Theorem \ref{reducibility}. Otherwise, decompose the space and the group into indecomposable components, and then also the corresponding weak $\varepsilon$-approximate fixed point $\xi$. By Lemma \ref{eigen}\textbf{(c)}
we choose an $i\in\{1, \ldots, k\}$ such that $\displaystyle \widetilde{\xi}_i= \frac{\xi_i} {||\xi_i||} \neq0$ is a weak $\widetilde{\varepsilon}$-approximate fixed point for $\mathcal{G}_{ii}$, where $\displaystyle \widetilde{\varepsilon} = \frac{n}{n_i}\varepsilon$, and $n_i$ is the dimension of the underlying space of the group $\mathcal{G}_{ii}$. We conclude {that}
\[
    \widetilde{\varepsilon}=\frac{n}{n_i}\varepsilon \leqslant \frac{n}{3600 n_i n^{11}} = \frac{1}{3600 n_i n^{10}} \leqslant\frac{1}{3600 n_i^{11}},
\]
so that $\mathcal{G}_{ii}$ has a common eigenvector by Theorem \ref{reducibility}, denoted by $\zeta$. Let $\widehat{\zeta}$ be the vector of the starting space whose $i$-th component equals $\zeta$ and the rest of whose components are zero. Clearly, this is a common eigenvector for $\mathcal{G}$.
 \end{proof}

\begin{remark}
  \emph{We will need the following observation about decompositions.
Assume a decomposition of a finite-dimensional vector space
$\mathcal{V}= \complex^n$ into a direct sum
$$
\mathcal{V}=\mathcal{V}_1\dotplus\mathcal{V}_2\dotplus\cdots \dotplus\mathcal{V}_r.
$$
Let a unitary group $\mathcal{N}$ be block diagonal with respect to this decomposition.
Then we may assume, with no loss of generality, that this sum is orthogonal. Indeed, choose any basis respecting this decomposition, and then orthogonalize this basis using the Gram-Schmidt process.  Clearly $\mathcal{N}$ is still upper-triangular with respect to the associated ``new", now orthogonal, decomposition
$$
\mathcal{V}= \mathcal{V}'_1 \oplus\mathcal{V}'_2\oplus\ldots \oplus \mathcal{V}'_r;
$$
since $\mathcal{N}$ is unitary we conclude that $\mathcal{N}$ must actually be block-diagonal with respect to this orthogonal decomposition as well. }
\end{remark}

\section{General finite groups}\label{sec:gen}

In the proof of the following theorem we need Clifford's theorem. A version of it that is close to our point of view is given in \cite[Theorem 2.3]{MasRad}.

\begin{theorem}\label{general finite group}
Let $n\in\mathbb{N}$ and $\varepsilon>0$ be given. If $n\geqslant 2$ and $\displaystyle\varepsilon<\frac{1}{3600 n^{11}}$, then every finite unitary  group $\mathcal{G}\subseteq M_n(\complex)$ with a weak $\varepsilon$-approximate fixed point is reducible.
\end{theorem}

 \begin{proof}
   We will prove this by contradiction. Assume the contrary, so that there exist $n\geqslant2$, $\varepsilon>0$ with $\displaystyle\varepsilon <\frac{1}{3600 n^{11}}$, and an irreducible finite unitary group $\mathcal{G}\ \subseteq M_n(\complex)$ with a weak $\varepsilon$-approximate fixed point $\xi$. Choose the smallest possible $n$ for which such a group exists. From among all such groups choose one of the smallest possible order, and call it $\mathcal{G}$.

We want to show that every proper normal subgroup $\mathcal{N}$ of $\mathcal{G}$ consists of scalars only. Since $\mathcal{N}$ is strictly smaller than $\mathcal{G}$ and has $\xi$ as a weak $\varepsilon$-approximate fixed point, it is reducible by assumption, so we can use \cite[Theorem 2.3]{MasRad}. Then $\mathcal{V}=\complex^n$ decomposes into a direct sum
$$
\mathcal{V}=\mathcal{V}_1\oplus\mathcal{V}_2\oplus\cdots\oplus\mathcal{V}_r,
$$
of invariant subspaces for $\mathcal{N}$ maximal with the property (when translated from the group-representations language) that for any $i=1,\ldots,r$ and for any $\mathcal{N}$-irreducible subspaces $\mathcal{U}, \mathcal{U}'\subseteq \mathcal{V}_i$ the restrictions $\mathcal{N} |_{\mathcal{U}}$ and $\mathcal{N}|_{\mathcal{U}'}$ viewed as groups of linear maps from $\mathcal{U}$ to itself and from $\mathcal{U}'$ to itself are spatially isomorphic. By the cited theorem the spaces $\mathcal{V}_i$ are of equal dimension, necessarily $\widetilde{n}= \dfrac{n}{r}$. Furthermore, since $\mathcal{N}$ is unitary, we can assume that the decomposition above is orthogonal by the remark of Section \ref{sec:decomp}.
%Namely, by Clifford's theorem $\complex^n$ decomposes into a direct sum of blocks
%\begin{equation}\label{eq:decomposition}
    %\complex^n=\mathcal{V}_1 \oplus\mathcal{V}_2\oplus \cdots \oplus \mathcal{V}_r
%\end{equation}
%of equal size, all invariant under $\mathcal{N}$. According to \cite[Theorem 2.3]{MasRad} $\mathcal{V}_1$, say, is maximal such that all irreducible sub-representations of $\mathcal{N}$ on $\mathcal{V}_1$ are isomorphic (as representations). Now, since the group $\mathcal{N}$ is made of unitaries, the orthogonal complement $\mathcal{V}_1^\bot$ of $\mathcal{V}_1$ is also invariant under $\mathcal{N}$. However, every irreducible representation of $\mathcal{N}$ on $\complex^n$ that is not isomorphic to irreducible sub-representations of  $\mathcal{V}_1$ is contained in some $\mathcal{V}_j$ for $j\neq1$, so that $\mathcal{V}_2 \oplus \cdots \oplus \mathcal{V}_r\subseteq \mathcal{V}_1^\bot$. Next, exchange in these considerations $\mathcal{V}_1$ with any $\mathcal{V}_j$ for $j\neq1$ to conclude that Decomposition \eqref{eq:decomposition} is actually orthogonal.
%Next, we decompose \eqref{eq:decomposition} even further. We decompose each of $\mathcal{V}_i$ into a direct sum of subspaces, invariant for $\mathcal{N}$, such that the restriction of $\mathcal{N}$ to it is irreducible. This can clearly be  done in an orthogonal way. (Choose the first subspace arbitrary, the second one in its relative orthogonal complement, and so on.) Without risk of confusion we denote the final decomposition again by \eqref{eq:decomposition}.
We decompose $\xi$ according to this decomposition denoting the components by $\xi_i$ and choose an $i\in\{1, \ldots, k\}$ by Lemma \ref{eigen}\textbf{(c)} such that $\displaystyle \widetilde{\xi}= \frac{\xi_i} {||\xi_i||} \neq0$ is a weak $\tilde{\varepsilon}$-approximate fixed point for $\widetilde{\mathcal{G}} =\mathcal{G}_i$, where $\displaystyle \tilde{\varepsilon}= \frac{n}{n_i}\varepsilon = r\varepsilon$. Since
\[
    \widetilde{\varepsilon}=r\varepsilon = \frac{n}{3600 \widetilde{n} n^{11}} = \frac{1}{3600 \widetilde{n} n^{10}} \leqslant \frac{1}{3600 \widetilde{n}^{11}},
\]
we have found a group $\widetilde{\mathcal{G}}$ that satisfies the assumptions of the theorem with $\widetilde{n}$ and $\widetilde{\varepsilon}$ instead of $n$ and $\varepsilon$. The possibility $n>\widetilde{n}$ contradicts the starting assumption of this proof, thus proving that $r=1$.

The decomposition of $\mathcal{V}$ into irreducible invariant subspaces of $\mathcal{N}$
\[
    \mathcal{V}= \mathcal{U}_1 \oplus\mathcal{U}_2\oplus\ldots \oplus \mathcal{U}_s,
\]
may be assumed orthogonal with no loss as above. Observe that the groups $\mathcal{N}_i= \mathcal{N} |_{\mathcal{U}_i}$ are all spatially isomorphic. Since they are strictly smaller than $\mathcal{G}$ and satisfy the assumptions of the theorem, these blocks are reducible by the starting assumption of this proof, so that $n/s=1$. Consequently, the spatially isomorphic groups $\mathcal{N}_i$ act on one-dimensional spaces implying that $\mathcal{N}$ is scalar.

%If $r=1$ all the irreducible representations of $\mathcal{N}$ are pairwise isomorphic and one-dimensional by Proposition \ref{connected group}, so that $\mathcal{N}$ is made of scalar matrices contradicting the assumption of this case. So, $r>1$ and $\mathcal{G}$ is block-monomial with respect to the decomposition \eqref{eq:decomposition}. Namely, by Clifford's theorem $\complex^n$ decomposes into a direct sum of irreducible blocks $\complex^n=\mathcal{V}_1 \oplus \cdots \oplus \mathcal{V}_r$ of equal size, all invariant under $\mathcal{N}$%, such that $\mathcal{G}$ is block monomial with respect to this decomposition.

We know now that every proper normal subgroup $\mathcal{N}$ of $\mathcal{G}$ consists of scalars only. If for such a maximal $\mathcal{N}$, the group $\mathcal{G}/\mathcal{N}$ is commutative, then by Suprunenko's theorem \cite[Proposition 3.1]{MasRad} (cf.\ also \cite[Theorem 24, p.\ 60]{Sup}) $\mathcal{G}$ is unitarily monomializable, and we come to a contradiction with the irreducibility of $\mathcal{G}$ using Theorem \ref{monomial group}. It remains to consider the case that $\mathcal{G}/\mathcal{N}$ is a noncommutative simple group. The longstanding Ore conjecture in which every element of every finite (non-abelian) simple group is a commutator, was proved to be true in \cite{LieOBrShaTie}. We apply this result to this quotient group to see that every element $G\in\mathcal{G}$ is of the form
$G=\lambda[A,B]$ for some $A,B\in\mathcal{G}$ and a scalar $\lambda$. This brings us to a contradiction with the irreducibility of $\mathcal{G}$ by Theorem \ref{product with scalar}, and we are done.
 \end{proof}

\section{ Connected groups of unitary matrices }\label{sec:con}

\begin{lemma}\label{connected} If $\mathcal{G}$ is a connected compact group of unitary matrices, then every element of its derived subgroup $\mathcal{G}'$, i.e.\ the closed subgroup generated by its commutators, is a commutator of two elements of $\mathcal{G}$.
\end{lemma}

 \begin{proof}
   Since $\mathcal{G}'$ is a semisimple (cf.\ \cite[Corollary 20.5.5]{TaYu}) connected compact Lie group, the claim is an immediate consequence of \S 4, No.\ 5, Proposition 10, Corollary of \cite{Bou}.
 \end{proof}

\begin{proposition}\label{connected group}
Let $n\in\mathbb{N}$ and $\varepsilon>0$ be given. If $n\geqslant 2$ and $\displaystyle\varepsilon<\frac{1}{32},%\frac{1}{3000 n^{9}}
$ then every connected compact unitary group $\mathcal{G}\subseteq M_n(\complex)$ with a weak $\varepsilon$-approximate fixed point is reducible.
\end{proposition}

\begin{proof}
  Since commutative groups are always reducible, we may assume with no loss that $\mathcal{G}$ is non-commutative. By Lemma \ref{connected} every element of $\mathcal{G}'$ is a commutator of two elements of $\mathcal{G}$. So, by Corollary \ref{bla}, $\mathcal{G}'$ has a fixed point, say $\xi$. Denote by $\mathcal{A}$ the linear span of $\mathcal{G}$, which is the same as the algebra generated by $\mathcal{G}$. Since $\mathcal{G}'$ is a normal subgroup of $\mathcal{G}$, we have that for any three elements $G,H,A\in\mathcal{G}$ the element $A^*H^*G^*HGA$ belongs to $\mathcal{G}'$, so that it fixes the point $\xi$. It follows that
\begin{equation}\label{commutator}
    GHA\xi=HGA\xi.
\end{equation}
This relation is then true for all $A\in\mathcal{A}$. Note that
\[
    \mathcal{V}:=\{A\xi\,|\,A\in\mathcal{A}\}
\]
is a non-trivial subspace of $\complex^n$, invariant for $\mathcal{G}$. If we prove that it is a proper subspace, we are done. Assume the contrary. Then, any two elements $G,H\in\mathcal{G}$ commute on $\mathcal{V}=\complex^n$, and $\mathcal{G}$ is a commutative group, contradicting the above.
\end{proof}

\begin{corollary}\label{lemma0}
  Let $n\in\mathbb{N}$ and $\varepsilon>0$ be given. If $n\geqslant 2$ and $\displaystyle\varepsilon<\frac{1}{32n}$ then every connected compact unitary group $\mathcal{G}\subseteq M_n(\complex)$ with a weak $\varepsilon$-approximate fixed point has a common eigenvector.
\end{corollary}

\begin{proof}
  By Proposition \ref{connected group} the group is reducible and we are able to decompose the space, the group, and the corresponding weak $\varepsilon$-approximate fixed point $\xi$ as in Section \ref{sec:decomp}. By Lemma \ref{eigen}\textbf{(c)} there always exists an $i\in\{1, \ldots, k\}$ such that $\widetilde{\xi}_i= \dfrac{\xi_i} {||\xi_i||} \neq0$ is a weak $\widetilde{\varepsilon}$-approximate fixed point for $\mathcal{G}_{ii}$, where $\widetilde{\varepsilon}= \dfrac{n}{n_i} \varepsilon\leqslant n \dfrac{1}{32n}= \dfrac{1}{32}$. So, by Proposition \ref{connected group}, such a $\mathcal{G}_{ii}$ acts on a one-dimensinal space, whose orthonormal basis vector we denote by $\zeta$. Let $\widehat{\zeta}$ be the vector of the starting space whose $i$-th component equals $\zeta$ and the rest of whose components are zero. It is clear that this is a common eigenvector for $\mathcal{G}$.
\end{proof}

\begin{lemma}\label{lemma2}
  Let a compact unitary group $\mathcal{G}\subseteq M_n(\complex)$ be given such that $\mathcal{G}=\overline{\mathbb{T}\mathcal{G}}$ and $\mathcal{G}=\mathcal{G}_0\mathcal{N}$, where $\mathcal{N}$ is a normal connected subgroup and $\mathcal{G}_0$ is a finite subgroup of $\mathcal{G}$. If $n\geqslant 2$, $\displaystyle\varepsilon \leqslant\frac{1}{3600 n^{11}}$, and $\mathcal{G}$ has a weak $\varepsilon$-approximate fixed point $\xi$, then it is reducible.
\end{lemma}

\begin{proof}
Towards a contradiction we assume $\mathcal{G}$ is irreducible. By Clifford's theorem $\mathcal{V}=\complex^n$ decomposes into a direct sum %(orthogonal with no loss)
$$
\mathcal{V}=\mathcal{V}_1\oplus\mathcal{V}_2\oplus\cdots\oplus\mathcal{V}_r,
$$
of invariant subspaces for $\mathcal{N}$ maximal such that for all $i=1, \ldots,r$ the restrictions of $\mathcal{N}$ to $\mathcal{N}$-irreducible subspaces of $\mathcal{V}_i$ are spatially isomorphic. Observe that this definition makes each of the spaces $\mathcal{V}_i$ unique, and since they could have been assumed orthogonal to each other by the Remark of Section \ref{sec:decomp}, this decomposition is indeed orthogonal. Also, the spaces $\mathcal{V}_i$ are of equal dimension, necessarily $\widetilde{n}= \displaystyle\frac{n}{r}$, and all of their $\mathcal{N}$-irreducible subspaces are of equal dimension.
We also decompose $\xi$ (which is necessarily a weak $\varepsilon$-approximate fixed point for the connected group $\mathcal{N}$), denoting its components by $\xi_i$, and choose an $i\in\{1, \ldots, k\}$ by Lemma \ref{eigen}\textbf{(c)} such that $\displaystyle \widetilde{\xi}_i= \frac{\xi_i} {||\xi_i||} \neq0$ is a weak $\widetilde{\varepsilon}$-approximate fixed point for $\mathcal{N} |_{\mathcal{V}_i}$, where $\widetilde{\varepsilon}= \dfrac{n}{n_i} \varepsilon = r\varepsilon$. Thus, $\mathcal{N} |_{\mathcal{V}_i}$ has a common eigenvector by Corollary \ref{lemma0}, so that the dimension of an $\mathcal{N}$-irreducible subspace in $\mathcal{V}_i$ (and consequently everywhere) is one. It follows that $\mathcal{N} |_{\mathcal{V}_i}$ {consists of scalar multiples of identity.} Returning to Clifford's theorem we observe that $\mathcal{G}$ is block-monomial with respect to this decomposition, and that groups $\mathcal{G}_{ii}$ viewed as groups of linear mappings from $\mathcal{V}_i$ to itself are irreducible. It is well-known (from, say, \cite[Chapitre 2.6]{Ser}) that the fact that our group is irreducible and unitary forces this decomposition to be automatically orthogonal.

We will now show that both possibilities $\widetilde{n}=1$ and $\widetilde{n}>1$ contradict the assumption that $\mathcal{G}$ is irreducible.

If $\widetilde{n}=1$, then $\mathcal{G}$ is a monomial unitary group having a weak $\varepsilon$-fixed point and is therefore reducible by Theorem \ref{monomial group}.
%%% Tu je sedaj popravek
%%% with blocks being simultaneously diagonal. So, $\mathcal{G}$ is reducible contradicting the starting assumption unless $n=r$. However, in this remaining case $\mathcal{G}$ is monomial and we are done by Theorem \ref{reducibility}.
Now assume that $\widetilde{n}>1$.  Then by Lemma \ref{eigen} there exists an $i$ such that $\widetilde{\xi_i}=\dfrac{\xi_i}{||\xi_i||}$ is a weak $\varepsilon_i$-fixed point for the irreducible group $\mathcal{H}=\mathcal{G}_{ii}$, where $\varepsilon_i<r\varepsilon$.  Note that the group $\mathcal{H}_0=\{G|_{\mathcal{V}_i}\,|\, G\in\mathcal{G}_0, G(\mathcal{V}_i)\subseteq\mathcal{V}_i\}$ is finite and that $\mathcal{H}=\mathbb{T}\mathcal{H}_0$.    Since $\mathcal{H}$ is irreducible, we have that $\mathcal{H}_0$ is also irreducible.  But this contradicts Theorem \ref{general finite group} (since $\mathcal{H}_0$ has a weak $\varepsilon_i$-approximate fixed point).
\end{proof}

\begin{theorem}\label{general group}
Let $n\in\mathbb{N}$ and $\varepsilon>0$ be given. If $n\geqslant 2$ and $\displaystyle\varepsilon<\frac{1}{3600 n^{11}}$, then every compact unitary group $\mathcal{G}\subseteq M_n(\complex)$ with a weak $\varepsilon$-approximate fixed point is reducible.
\end{theorem}

 \begin{proof}
   We assume with no loss of generality that $\mathcal{G}= \overline{\mathbb{T}\mathcal{G}}$ and denote by $\mathcal{N}$ the connected component of the identity. {It is well known that there is a finite group $\mathcal{G}_0$ such that $\mathcal{G}=\mathcal{G}_0 \mathcal{N}$ \cite{BerGurMas}, and hence we are done by Lemma \ref{lemma2}}.
 \end{proof}

\section{Common eigenvectors}\label{sec:eigenvec}

In this section we present the main results of the paper. Let $\mathcal{G}$ be a group of unitary $n\times n$ matrices with a weak $\varepsilon$-approximate point $\xi$, with $\|\xi\| = 1$ and $0<\varepsilon<1$ as assumed throughout the paper. The following theorem follows easily from Theorem \ref{general group} and Lemma \ref{eigen}.
\begin{theorem}\label{common}
  Let $n\geqslant 2$ and $\displaystyle0<\varepsilon<\frac{1}{3600 n^{11}}$.  If a group $\mathcal{G}$ of $n\times n$ unitary matrices has a weak $\varepsilon$-approximate point, then $\mathcal{G}$ has a common eigenvector.
\end{theorem}

\begin{proof}
  Assume, using a unitary similarity if necessary, that $\mathcal{G}$ is block diagonal with irreducible blocks. For the corresponding decomposition we recall the notation of Section \ref{sec:decomp}. By Lemma \ref{eigen}\textbf{(c)} we choose an $i\in\{1, \ldots, k\}$ such that $\displaystyle \widetilde{\xi}_i= \frac{\xi_i} {||\xi_i||} \neq0$ is a weak $\widetilde{\varepsilon}$-approximate fixed point for $\mathcal{G}_i$, where $\displaystyle $.
   \[
   \widetilde{\varepsilon} = \frac{n}{n_i}\varepsilon \leqslant \frac{1}{3600 n_i^{11}}.
   \]
  In case $n_i\geqslant 2$, this would mean that $\mathcal{G}_i$ is reducible by Theorem \ref{general group} which is impossible.  Hence $n_i=1$ and therefore $\xi_i$ is a common eigenvector for $\mathcal{G}$.
\end{proof}

\begin{lemma}\label{last_lemma}
  Let  $x>y>0$ , $\displaystyle\alpha\in \left[\frac{2\pi}{3}, \frac{4\pi}{3}\right]$, and let $\xi=e^{\alpha i}$.  Then $|x+\xi y| \leqslant x$.
\end{lemma}

 \begin{proof}
   With no loss assume that $\displaystyle \alpha\in \left[\pi, \frac{4\pi}{3} \right]$ (otherwise replace $\xi$ by $\overline{\xi}$).   Consider the triangle in $\complex$ with vertices $0$, $x$ and $-\xi y$.  The lengths of its sides are $x,y,$ and $z=|x+\xi y|$.  Since the angle opposite $z$ is $\displaystyle\alpha-\pi\leqslant\frac{\pi}{3}$, we have that $z$ cannot be the longest side of the triangle.
 \end{proof}

\begin{theorem}\label{main}
  Let $n\geqslant 2$, $\displaystyle0<\varepsilon<\frac{1}{3600 n^{11}}$, and let a group of $n\times n$ unitary matrices $\mathcal{G}$ have a weak $\varepsilon$-approximate fixed point $\xi$.  Then there exists a common eigenvector $\eta$ such that $||\xi-\eta||^2<3600 n^{11} \varepsilon$.
\end{theorem}

\begin{proof}
  With no loss we assume that $\mathcal{G}=\overline{\mathbb{T}\mathcal{G}}$.  By Theorem \ref{common} $\mathcal{G}$ has a common eigenvector. Let $r$ be the greatest number of linearly independent eigenvectors of $\mathcal{G}$.  With no loss of generality we assume that $\mathcal{G}$ is block diagonal with $k=r+1$ blocks, where the first $r$ blocks are $1\times 1$ and the $k$-th block has no common eigenvector (or that the $k$-th block is of size $0\times 0$). For this decomposition recall the notation of Section \ref{sec:decomp}. If $\xi_k=0$, then $\xi$ itself is a common eigenvector and we are done. If $\xi_k\not=0$, then $\displaystyle \tilde{\xi}_k=\frac{\xi_k}{||\xi_k||}$ is a weak $\varepsilon_k$-approximate fixed point for $\mathcal{G}_k$. Hence $\displaystyle\varepsilon_k\geqslant \frac{1}{3600(n-r)^{11}}$ by Theorem \ref{common} and therefore $||\xi_k||^2\leqslant 3600(n-r)^{11} \varepsilon$.

For $a_i=\|\xi_i\|^2, i\in\{1,2,\ldots,r\},$ assume with no loss that $a_1\geqslant a_2\geqslant \ldots \geqslant a_r$. Let $G\in\mathcal{G}$, and suppose, if possible, that for some $i\in\{2,\ldots, r\}$ we have that $G_i\not= G_1$.  Let $m$ be an integer such that an argument $\alpha$ of $(G_i\overline{G_1})^m$ is between $\displaystyle \frac{2\pi}{3}$ and $\displaystyle \frac{4\pi}{3}$ and let $H= (\overline{G_1}G)^m$ so that $H_1=1$ and an argument of $H_i$ is between  $\displaystyle \frac{2\pi}{3}$ and $\displaystyle \frac{4\pi}{3}$. Using Lemma \ref{last_lemma} we then get
\begin{eqnarray*}
1-\varepsilon &<& |\langle H\xi,\xi\rangle| \leqslant |\langle H_1\xi_1,\xi_1\rangle +\langle H_i\xi_i,\xi_i\rangle|+\sum_{j\not\in\{1,i\}} ||\xi_j||^2 \\
&= & |a_1+H_i a_i| + (1-a_1-a_i) \le a_1+(1-a_1-a_i) = 1-a_i,
\end{eqnarray*}
and therefore $a_i<\varepsilon$.  Let $s$ be the largest integer such that $a_s\geqslant \varepsilon$ (clearly $s\geqslant 1$) and let $\eta=\xi_1 +\ldots +\xi_s$.  Now note that $\eta$ is a common eigenvector for $\mathcal{G}$ and that
\[||\xi-\eta||^2=\sum_{j=s+1}^k ||\xi_j||^2\leqslant (r-1)\varepsilon + 3600(n-r)^{11}\varepsilon < 3600 n^{11} \varepsilon.\]
\end{proof}

\vskip 5mm

{\bf Acknowledgement.}
The authors are thankful to Dragomir \v Z.\ Djokovi\' c for a discussion of the result presented in Lemma \ref{connected} and the reference given in its proof.

%%%%%%%%%%%%%%%%%%%%%%%%%%%%%%%%%%%%%%%%%%%%%%%%%%%%%%%%%%%%%
%\begin{thebibliography}{1}
%#\bibitem{schneider-84}
%#Hans Schneider.
%#\newblock  Theorems on M-splittings of a singular M-matrix which           depend on graph structure.
%#\newblock  {\em Linear Algebra and its Applications}, 58:407--424, 1984.

%\bibitem{varga-62}
%Richard S.~Varga.
%\newblock  {\em Matrix Iterative Analysis}.
%\newblock  Prentice-Hall, Englewood Cliffs, New Jersey, 1962.

%\bibitem{ela-paper}
%S.~Friedland and H.~Schneider. 
%\newblock  Spectra of expansion graphs. 
%\newblock  {\em Electronic Journal of Linear Algebra}, 6:2--10, 1999.

%\end{thebibliography}

\end{document}